\documentclass[12pt]{amsart}
\usepackage{amsmath}
\usepackage{amssymb}
\usepackage{amsthm}
\usepackage{amscd}
\usepackage{amsfonts}
\usepackage{graphicx}
\usepackage{fancyhdr}
\usepackage{epsfig}
\usepackage{tikz}
\usepackage{setspace}
\usepackage{verbatim}
\usepackage{subfigure}
\usepackage{cleveref}
\usepackage{tikz}
\usepackage{enumitem}

\usetikzlibrary{patterns}
\usepackage{amsmath, times}
\theoremstyle{plain} \numberwithin{equation}{section}
\newtheorem{theorem}{Theorem}[section]
\newtheorem{corollary}[theorem]{Corollary}

\newtheorem{lemma}[theorem]{Lemma}
\newtheorem{proposition}[theorem]{Proposition}

\theoremstyle{definition}
\newtheorem{definition}[theorem]{Definition}

\newtheorem{remark}[theorem]{Remark}
\newtheorem{example}[theorem]{Example}
\newtheorem{algorithm}[theorem]{Algorithm}

\DeclareMathOperator{\support}{supp }

\DeclareMathOperator{\mi}{mi}
\DeclareMathOperator{\atoms}{atoms}

\DeclareMathOperator{\pd}{pd}

\title{Lattices and Hypergraphs associated to square-free monomial ideals}

\author[Lin]{Kuei-Nuan Lin}
\address{Department of Mathematics, The Penn State University, Greater Allegheny Campus,  McKeesport, PA}
\email{kul20@psu.edu}

\author[Mapes]{Sonja Mapes}
\address{Department of Mathematics, University of Notre Dame,  Notre Dame, IN}
\email{smapes1@nd.edu}


\date{\today}

\begin{document}
\maketitle

\begin{abstract}
Given a square-free monomial ideal $I$ in a polynomial ring $R$ over a
field $\mathbb{K}$,  one can associate it with its LCM-lattice and its
hypergraph. In this short note, we establish the connection between
the LCM-lattice and the hypergraph, and in doing so we provide a
sufficient condition for removing higher dimension edges of the
hypergraph without impacting the projective dimension of the
square-free monomial ideal. We also offer algorithms to compute the
projective dimension of a class of square-free monomial ideals built
using the new result and previous results of Lin-Mantero. 
\end{abstract}

\section{Introduction}

Finding the projective dimension or the Castelnuovo-Mumford regularity
of a homogeneous ideal, $I$, in a graded polynomial ring $R=\mathbb{K}[x_{1},...,x_{n}]$
over a field $\mathbb{K}$ has been an active research field over the last
decades.  See, for example, the survey papers \cite{Ha} and \cite{MS}. These two invariants
give important information about the ideal, and they measure the complexity
of the ideal. Moreover they play important roles in algebraic geometry, commutative
algebra, and combinatorial algebra. In general, one finds the graded
minimal free resolution of an ideal to obtain those invariants, but
the computation can be difficult and computationally
expensive. Alternatively one can try finding bounds for these invariants
using properties of the ideal.  Studying monomial
ideals and specifically square-free monomial ideals is important in
this strategy.  In particular, it is well known that the regularity of a given ideal is bounded by the regularity of
its initial ideal (see for example Theorem 22.9 \cite{Peeva}.) The polarization of a monomial ideal does not change
its projective dimension ( see for example Theorem 21.10 \cite{Peeva}), hence one may use square-free
monomial ideals to understand projective dimensions of monomial ideals
in general.  Finally when $I$ is a
square-free monomial ideal, there is a dual relation between the projective
dimension and the regularity with respect to the Alexander dual \cite{Terai}. Thus, finding
the projective dimension of a square-free monomial ideal is a
central problem in this field, see for instance \cite{DS}. Additionally, one can use the projective dimension to decide whether an ideal is Cohen-Macaulay.  

This paper focuses on using two combinatorial objects associated to a square-free monomial ideal in place of the minimal free
resolution:  the dual hypergraph and the LCM-lattice. Kimura, Terai
and Yoshida define the dual hypergraph of a square-free monomial ideal in order to compute
its arithmetical rank \cite{KTY}. Since
then, there are a couple of papers using this combinatorial object to
study various properties, for example, \cite{HaLin} and \cite{LMc}. In
particular, Lin and
Mantero use it to show that ideals with the same dual hypergraph
have the same  total Betti numbers and projective dimension \cite{LMa1}, which has found use in
other papers, such as in \cite{KMa}. 

For the second combinatorial object, Gasharov, Peeva,
and Welker define the LCM-lattice of a monomial ideal. They show that if
there is a map between two LCM-lattices which is a bijection on the
atoms and preserves joins, then a resolution of an ideal in the domain is the resolution
of the other with respect to the map, i.e. when the map is an isomorphism those two ideals have the
same total Betti numbers and projective dimension \cite{GPW}. Phan and Mapes
show that every finite atomic lattice is the LCM-lattice of a monomial
ideal via a special construction, \cite{mapes} and \cite{phan}. It is natural to inquire if there
is a connection between the dual hypergraph and the LCM-lattice of
a given square-free monomial ideal. The positive answer is one of
the first results in this paper (\Cref{ThmLCMHG}).  Specifically, one can construct the
LCM-lattice of a monomial ideal via its dual hypergraph and vice versa
as shown in \Cref{LCM-H}.

The results in \cite{LMa1} and \cite{LMa2} focus mostly on
determining the projective dimension when the dual hypergraph of an
ideal consists only of vertices and edges with cardinality 2 (i.e. is
``1-dimensional'').  Moreover the work of Kimura, Rinaldo, and Terai shows that the projective dimension
of a monomial ideal depends on the 1-skeleton structure of the dual
hypergraph \cite{KRT}. It is clear that sometimes a higher dimensional
edge of a dual hypergraph can be removed without impacting the projective
dimension of a monomial ideal. This paper focuses mostly on the
question:  Under what conditions can one
remove higher dimensional edges without changing the projective dimension
of a hypergraph?  The work by Lin and Mantero answers part of this question \cite{LMa2} with some restrictions on
the 1-skeleton of the hypergraph. 
In this paper, using the connection to LCM-lattices, we show a sufficient condition for when removing the
higher dimensional edge has no impact on total Betti numbers and hence the projective dimension (\Cref{sameTotalBetti}). We explain why the work of Kimura, Rinaldo, and Terai shows that the projective dimension of a monomial ideal depends on the 1-skeleton structure of the dual hypergraph \cite{KRT} and explain the result in Lin and Mantero in a combinatorial construction (\Cref{explain}). We then proceed with our
results concerning higher dimensional edges on bushes in 
\Cref{BushesSec} . In the end, we provide algorithms for computing the projective dimension of certain square-free monomial ideals using hypergraphs without the computation of minimal free resolution of the square-free monomials.
Through out this paper, ideals are square-free monomial ideals in a polynomial ring $R$ over the field $\mathbb{K}$.

\section{Lattices and LCM-lattices}
A \emph{lattice} is a set $(P, <)$  with an order relation $<$, which is transitive and antisymmetric satisfying the following properties:
\begin{enumerate}
\item $P$ has a maximum element denoted by $\hat{1}$
\item $P$ has a minimum element denoted by  $\hat{0}$
\item Every pair of elements $a$ and $b$ in $P$ has a join $a \vee b$, which is the least upper bound of the two elements
\item Every pair of elements $a$ and $b$ in $P$ has a meet $a \wedge b$, which is the greatest lower bound of the two elements.  
\end{enumerate}

We define an \emph{atom} of a lattice $P$ to be an element $x \in P$
such that $x$ covers $\hat{0}$ (i.e. $x > \hat{0}$ and there is no
element $a$ such that $x > a > \hat{0}$). We will denote the set of
atoms as $\atoms (P)$.  

\begin{definition}
If $P$ is a lattice and every element in $P -\{\hat{0}\}$ is the join of atoms, then $P$ is an \emph{atomic lattice}.  Further, if $P$ is finite, then it is a \emph{finite atomic lattice}. 
\end{definition}

Given a lattice $P$,  elements $x \in P$ are \emph{meet-irreducible} if $x \neq a \wedge b$ for any $ a > x, b>x$. 
The set of meet-irreducible elements in $P$ is denoted by  $\mathrm{mi}(P)$. 
Given an element $x \in P$, the \emph{filter} of $x$ is $\lceil {x} \rceil = \{a \in P
| x \leqslant a\}$.  

\begin{remark}\label{meetIrredMakeLattice}
Lemma 2.3 in \cite{mapes} guarantees that if $P$ is a finite atomic
lattice, then every element $p$ in $P-\{\hat{1}\}$ is the meet of all the
meet irreducible elements greater than $p$.  
\end{remark}

For the purposes of this paper it will often be convenient to consider finite atomic lattices as sets of sets in the following way.  Let $\mathcal{S}$ be a set of subsets of $\{1,...,n\}$ with no duplicates, closed under intersections, and containing the entire set, the empty set, and the sets $\{i\}$ for all $1 \leqslant i \leqslant n$.  Then it is easy to see $\mathcal{S}$ is a finite atomic lattice by ordering the sets in $\mathcal{S}$ by inclusion.  This set obviously has a minimal element, a maximal element, and $n$ atoms, so by {\cite[Proposition 3.3.1]{stanley}}, we need to show that it is a meet-semilattice.  Here the meet of two elements would be defined to be their intersection. Since $S$ is closed under intersections, this is a meet-semilattice.  Conversely, it is clear that all finite atomic lattices can be expressed in this way, simply by letting $$\mathcal{S}_P = \{ \sigma \,\vert\, \sigma = \support(p), p\in P\},$$ where $\support (p) = \{a_i \,\vert\, a_i \leqslant p, a_i \in \atoms(P)\}$.    

The only poset that we are interested in this paper is the LCM-lattice of a monomial ideal.  As a poset, the
LCM-lattice of $I$, typically denoted as $L_I$, is the set of all
least common multiples of subsets of generators of $I$ partially
ordered by divisibility.  It is easy to show that the LCM-lattice is
in fact a finite atomic lattice.  

\subsection{Coordinatizations of LCM-lattices}\label{coord}

LCM-lattices became important in the study of resolutions of monomial
ideals in the paper by Gasherov, Peeva, and Welker \cite{GPW}.  Two primary results, Theorem 3.3 and Theorem 2.1 in \cite{GPW}, respectively, will be important to us here:
 If one has monomial ideals $I$ and $I'$ in polynomial rings $R$
  and $R'$ with LCM-lattices $L$ and $L'$, respectively.  If
  there is a join preserving map $f: L \rightarrow L'$ which is a
  bijection on atoms then a minimal resolution of $R/I$ can be relabeled
  to be a resolution of $R'/I'$.  And if $f$ is an isomorphism then the
  relabeled resolution is a minimal resolution of $R'/I'$.

Continuing this study, one of the main results (Theorem 5.1) of \cite{phan} is to show that
every finite atomic lattice is in fact the LCM-lattice of a monomial
ideal.  This result was generalized by a modified construction in
\cite{mapes}, which also showed that with the modified
construction  all monomial ideals can be realized this way.  We
include a brief description of this work here for
the convenience of the reader.

Define a {\it labeling} of a finite atomic lattice $P$ as any assignment of non-trivial monomials $\mathcal {M} = \{m_{p_1}, ..., m_{p_t}\}$ to some elements $p_i \in P$.  It will be convenient to think of unlabeled elements as having the label $1$. Define the monomial ideal $M_{\mathcal{M}}$ to be the ideal generated by monomials
\begin{equation} \label{IdealDef}
x(a) = \prod_{ p \in \lceil{a}\rceil ^ c} m_p 
\end{equation} for each $a \in \atoms (P)$ where $\lceil a \rceil^c$ means take the complement of $\lceil a \rceil$ in $P$.  We say that the labeling $\mathcal{M}$ is a {\it coordinatization} if the LCM-lattice of $M_{\mathcal{M}}$ is isomorphic to   $P$. 

The following theorem, which is Theorem 3.2 in \cite{mapes}, gives a
criteria for when a labeling is a coordinatization.  

\begin{theorem}\label{coordinatizations}
Any labeling $\mathcal{M}$ of elements in a finite atomic lattice $P$ by monomials satisfying the following two conditions will yield a coordinatization of $P$.

\begin{enumerate}
\item[(C1)] If $p \in \mi (P)$ then $m_p \not = 1$.  (i.e. all meet-irreducibles are labeled)
\item[(C2)] If $\gcd (m_p, m_q) \not = 1$ for some $p, q \in P$ then $p$ and $q$ must be comparable.  (i.e. each variable only appears in monomials along one chain in $P$.)
\end{enumerate}
\end{theorem}

\begin{example}
In Figure \ref{FirstCoord} we see an example of a poset $P$ with a
labeling on the vertices.  We can see that this labeling satisfies
both conditions of \Cref{coordinatizations} and so one can check that the corresponding
monomial ideal $(bcd, abc, a^2c, a^2b)$ has $P$ as its LCM-lattice.
Note that this ideal is not square-free, to get a square-free monomial
ideal with this LCM-lattice, one would just need to replace one of the
$a$ labels with a new variable or square-free monomial which does not
use any of the variables $a,\dots, d$.  

\begin{figure}[h] 
\caption{}\label{FirstCoord}
\begin{center}
 \begin{tikzpicture}[scale=1, vertices/.style={draw, circle, inner sep=2pt}]
             \node [vertices] (0) at (-0+0,0){};
             \node [vertices] (1) at (-2.25+0,1.33333){$a$};
             \node [vertices] (2) at (-2.25+1.5,1.33333){};
             \node [vertices] (3) at (-2.25+3,1.33333){$b$};
             \node [vertices] (4) at (-2.25+4.5,1.33333){$c$};
             \node [vertices] (5) at (-.75+0,2.66667){$a$};
             \node [vertices] (6) at (-.75+1.5,2.66667){$d$};
             \node [vertices] (7) at (-0+0,4){};
     \foreach \to/\from in {0/1, 0/2, 0/3, 0/4, 1/5, 2/5, 2/6, 3/6, 4/6, 5/7, 6/7}
     \draw [-] (\to)--(\from);
     \end{tikzpicture}
\end{center}
\end{figure}

\end{example}

\section{Hypergraph of a square-free monomial ideal}

Kimura, Terai, and Yoshida associate a square-free monomial ideal with
a hypergraph in \cite{KTY}, see Definition \ref{HGDef}.  Note that
this construction is different from the construction, associating ideals
with hypergraphs, which is extended from the study of edge ideals.  In particular
relative to edge ideals, the hypergraph of Kimura, Terai, and Yoshida
might be more aptly named the ``dual hypergraph''.   The construction
of dual hypergraphs is first introduced by Berge in \cite{Berge}. In
the edge ideal case, one associates a square-free monomial with a
hypergraph by setting variables as vertices and each monomial
corresponds to an edge of the hypergraph (see for example
\cite{Ha}). In the following definition, we actually associate
variables with edges of the hypergraph and vertices with the monomial generators of the ideal, and in practice this is the dual hypergraph of the
hypergraph in the edge ideal construction.

\begin{definition}\label{HGDef}
Let $I$ be a square-free monomial ideal in a polynomial ring with $n$ variables with minimal monomial
generating set $\{m_1, \dots, m_\mu\}$.  Let $V$ be the set
$\{1,\dots, \mu\}$.  We define $\mathcal{H}(I)$ (or $\mathcal{H}$ when
$I$ is understood) to
be the hypergraph associated to $I$ which is defined as $\{\{j \in V :
x_i|m_j\} : i = 1,2,\dots, n\}$.   Moreover $\mathcal{H}$ is
\emph{separated} if in addition for every $1 \leq j_1 < j_2 \leq
\mu$, there exist edges $F_1$ and $F_2$ in $\mathcal{H}$, so that $j_1
\in F_1 \cap (V-F_2)$ and  $j_2
\in F_2 \cap (V-F_1)$
\end{definition}

Note that when a hypergraph is separated then its vertices correspond to a
minimal generating set of the monomial ideal. 

\begin{example} 
Let \[I=(abo,bcp,cdepq,efqr,fgr,ghr,hijoq,jk,klp,lmo,mn),\] the \Cref{FirstHG} is the hypergraph associated to $I$ via the \Cref{HGDef}

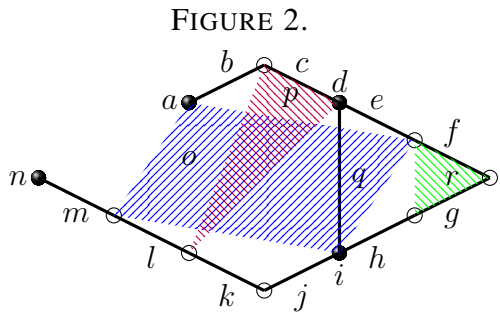
\begin{figure}[h] 
\caption{}\label{FirstHG}
\begin{center}
\begin{tikzpicture}
\shade [shading=ball, ball color=black]  (0,0) circle (.1) node [left] {$n$};
\draw  [shape=circle] (1,-0.5) circle (.1) ;
\draw  [shape=circle] (2,-1) circle (.1) ;
\draw  [shape=circle] (3,-1.5) circle (.1) ;
\shade [shading=ball, ball color=black] (4,-1) circle (.1) node [below] {$i$};
\draw  [shape=circle] (5,-0.5) circle (.1) ;
\draw  [shape=circle] (6,0) circle (.1) ;
\draw  [shape=circle] (5,0.5) circle (.1) ;
\shade [shading=ball, ball color=black] (4,1) circle (.1)  node [above] {$d$};
\draw  [shape=circle] (3,1.5) circle (.1) ;
\shade [shading=ball, ball color=black]  (2,1) circle (.1)  node [left] {$a$};

\draw [line width=1.2pt] (0,0)--(1,-0.5) node [pos=.5, below] {$m$};
\draw [line width=1.2pt] (1,-0.5)--(2,-1) node [pos=.5, below] {$l$};
\draw [line width=1.2pt] (2,-1)--(3,-1.5) node [pos=.5, below] {$k$};
\draw [line width=1.2pt] (3,-1.5)--(4,-1) node [pos=.5, below] {$j$};
\draw [line width=1.2pt ] (4,-1)--(4,1) node [pos=.5, right] {$q$};
\draw [line width=1.2pt ] (4,-1)--(5,-0.5) node [pos=.5, below] {$h$};
\draw [line width=1.2pt] (5,-0.5)--(6,0) node [pos=.5, below] {$g$};
\draw [line width=1.2pt] (6,0)--(5,0.5) node [pos=.5, above] {$f$};
\draw [line width=1.2pt] (5,0.5)--(4,1) node [pos=.5, above] {$e$};
\draw [line width=1.2pt] (4,1)--(3,1.5) node [pos=.5, above] {$c$};
\draw [line width=1.2pt] (3,1.5)--(2,1) node [pos=.5, above] {$b$};
\path [pattern=north east lines, pattern  color=blue]   (1,-0.5)--(4,-1)--(5,0.5)--(2,1)--cycle;
\path [pattern=north west lines, pattern  color=green]   (6,0)--(5,-0.5)--(5,0.5)--cycle;
\path [pattern=north west lines, pattern  color=purple]   (2,-1)--(3,1.5)--(4,1)--cycle;

\path (2,0.5)--(2,0.5) node [pos=.5, below ] {$o$ };
\path (3.2,1.35)--(3.5,1.35) node [pos=.5, below ] {$p$ };
\path (5,0)--(6,0) node [pos=.5 ] {$r$ };

\end{tikzpicture}
\end{center}
\end{figure}
\end{example}

Some important terminology regarding these hypergraphs is the
following.  We say a vertex $i \in V$ of $\mathcal{H}$ is an
\emph{open vertex} if $\{i\}$ is not in $\mathcal{H}$, and otherwise
$i$ is \emph{closed}.  In Figure \ref{FirstHG}, we can see that the
vertices labeled by $a, d, i,$ and $n$ are all closed, and the rest
are open.  Let  $\mathcal{H}^{i} = \{ F \in H : |F| \leq i +1\}$ denote the $i$-th dimensional subhypergraph of $\mathcal{H}$ where $|F|$ is the cardinality of the $F$. We call $\mathcal{H}^{1}$, the 1-skeleton of $\mathcal{H}$. We write  $\mathcal{H}^{1}_{O}$ as the subgraph of $\mathcal{H}^{1}$ such that it only has open vertices of $\mathcal{H}^{1}$.

Let $R = \mathbb{K}[x_1, \dots, x_n]$ be a polynomial ring over a field $\mathbb{K}$.  The minimal free resolution of $R/I$ for an ideal $I\subset R$ is an exact sequence of the form 
	\[
	0  \rightarrow  \bigoplus_j S(-j)^{\beta_{p,j}(R/I)}  \rightarrow \dots  \rightarrow  \bigoplus_j S^{\beta_{1,j}(R/I)}  \rightarrow  R \rightarrow  R/I  \rightarrow  0
		\]
The exponents $\beta_{i,j}(R/I)$ are invariants of $R/I$, called the Betti numbers of $R/I$. In general, finding Betti numbers is still a wide open question. This project focuses on studying the projective dimension of $R/I$, denoted $\pd(R/I)$, which is defined as follows
	\begin{align*}
		\pd(R/I) &= \max\{ i \mid \beta_{i,j}(R/I) \neq 0\}.
	\end{align*}

Recently there has been a number of results concerning determining the
projective dimension of square-free monomial ideals
from the associated hypergraph.   The proposition below
allows us to talk about the projective dimension of a
hypergraph rather than an ideal.

\begin{proposition}(Proposition 2.2 \cite{LMa1})\label{sameBetti} If $I_1$
  and $I_2$ are square-free monomial ideals associated to the
same separated hypergraph $\mathcal{H}$, then the total Betti numbers of two ideals coincide.
 \end{proposition}

 From now on, we will use $\pd(\mathcal{H}(I)) $ in
the place of $ \pd(R/I)$ throughout the paper.  If a hypergraph
$\mathcal{H}$ is an union of two disconnected hypergraphs $G_1$ and
$G_2$, we have $\pd(\mathcal{H})=\pd(G_1)+\pd(G_2)$ by Proposition
2.2.8 of \cite{Jacques}.

\section{Connection between dual hypergraphs and LCM-lattices}\label{LCM-H}

In this section, we want to show that one can re-build the LCM-lattice
$L_I$ of the monomial ideal $I$ from $\mathcal{H}(I)$.  Moreover, in
order to do so we will need to prove an important result that will allow us to
detect meet-irreducilbe elements of the LCM-lattice from the
hypergraph itself.  

First let us define a finite atomic lattice given a hypergraph
$\mathcal{H}$.  Thinking of a finite atomic lattice as a set of sets,
define $L_{\mathcal{H}}$ to be the meet-closure (or intersection-closure) of the set $\{F^c \, \vert \, F \in \mathcal{H}   \}$, where
$F^c$ means take the complement of each edge of $\mathcal{H}$ in the
set of vertices of $\mathcal{H}$.  This meet-closure will be a
meet-semilattice (partially ordered by inclusion), and so to make it a
lattice we add the set of all the vertices of $\mathcal{H}$ (i.e. a
maximal element).  

Our claim is that $L_{\mathcal{H}(I)} = L_I$, and we will prove this
by constructing a coordinatization of $L_{\mathcal{H}(I)}$ that will
produce the same monomial ideal $I$.  First though, we need to identify
which elements of $L_{\mathcal{H}(I)}$ are meet irreducible.      

Recall that a meet-irreducible of a finite atomic lattice $L$ is an element which is not the
meet of any 2 elements.  Thinking of $L$ as being a
set of subsets this means that there is an subset $\sigma$ in $L$
which is not the intersection of 2 (or more) subsets $\tau_1, \dots, \tau_t$ of $L$ where
none of these $\tau_i$ are $\sigma$.  If $L$ is the $L_{\mathcal{H}(I)}
  $ then taking complements this should
correspond to the following statement about edges in
$\mathcal{H}(I)$. 

\begin{proposition}\label{removeMeets}
If $F\in \mathcal{H}(I)$ is  the union
of 2 or more distinct edges of $\mathcal{H}(I)$ then the edge $F$
corresponds to an element which is a meet in $L_{\mathcal{H}(I)}$.
\end{proposition}

\begin{proof}
Suppose $F = \cup_{i=1}^{t} G_i$ where $G_i$ is also an edge of
$\mathcal{H}(I)$.  Then by De Morgan's Laws the corresponding elements
in $L_I$ are $F^c = \cap_{i=1}^t G_i^c$ where the notation $F^c$ means
complement in $\{1, \dots, \mu\}$.  In terms of thinking of $L_{\mathcal{H}(I)}$ as
set of sets closed under intersections, this says precisely that $F^c$
is the meet of $G_1^c, \dots, G_t^c$. 
\end{proof}

Now we are ready to show that the 2 lattices are actually the same.

\begin{theorem}\label{ThmLCMHG}
If $I$ is a square-free monomial ideal, then $L_{\mathcal{H}(I)} = L_I$. \end{theorem}

\begin{proof}
We begin by constructing a coordinatization for $L_{\mathcal{H}(I)}$,
which will hopefully produce the ideal $I$ as follows.  By Equation
\ref{IdealDef} we can see that if $F_i$ is an edge of $\mathcal{H}(I)$
corresponding to the variable $x_i$, then in $L_{\mathcal{H}(I)}$ we label the
element \[ \bigvee_{j\in [\mu]; j \not\in F_i} a_j = F_i^c \] (where $a_j$'s are the atoms of
$L_{\mathcal{H}(I)}$) with the variable $x_i$.  Note that here the
equality is a bit of an abuse of notation where on one side we are
thinking of elements as joins of atoms and on the other side we are
thinking of them as subsets of the vertex set. 

Note that this labeling by definition will satisfy
condition (C2) since each variable only gets used once, so it remains to
consider what condition (C1) means in this case. Now, consider the fact that condition (C1) requires that all meet
irreducibles of $L_{\mathcal{H}}(I)$ are labeled.  By \Cref{removeMeets}, we have a precise description of the meet
  irreducible elements as being a subset of the edges of
  $\mathcal{H}(I)$.  As we are labeling all elements of
  $\mathcal{H}(I)$, condition (C1) is satisfied and so the labeling we
  have given is in fact a coordinatization.  

Now if we can show that the coordinatization we produced in fact gives
the ideal $I$ then we will know that the lattice we coordinated is
in fact the LCM-lattice of $I$, thus proving the theorem.  By
construction, the monomial associated to the atom $a_i$ will be
the product of the variables
corresponding to the edges that contain $i$, which is precisely the
ideal $I$.

\end{proof}

This relationship between the LCM-lattice and the hypergraph is best
seen in the following example.  

\begin{example} \label{LCMHyper}
This example gives the relationship between the LCM-lattice and the hypergraph. Let $I=(ab,bcg,cdg,de,efg)=(f_1,f_2,f_3,f_4,f_5)$, then \Cref{LCMHG} is the hypergraph of the $I$ such that  $\mathcal{H}(I)=\{\{1\}=F_a,\{1,2\}=F_b,\{2,3\}=F_c,\{3,4\}=F_d,\{4,5\}=F_e,\{5\}=F_f,\{2,3,5\}=F_g\}$ . We take the compliment of  $\mathcal{H}(I)$,  which is $\{\{2,3,4,5\},\{3,4,5\},\{1,4,5\},\{1,2,5\},\{1,2,3\},\{1,2,3,4\},\{1,4\}\}$, and we find intersection of all edges to obtain $\{\{1\},\{2\},\{3\},\{4\},\{5\},\{1,2\},\{1,4\}=\{g\},\{1,5\},\{2,3\},$ $\{2,5\},\{3,4\},\{4,5\},\{1,2,3\}=\{e\},\{1,2,5\}=\{d\},\{1,4,5\}=\{c\},\{2,3,4\},\{3,4,5\}$ $=\{b\},\{1,2,3,4\}=\{f\},\{2,3,4,5\}=\{a\},\{1,2,3,4,5\}\}$. \Cref{LCMHG2} shows the LCM-lattice of $I$ and the connection.

\begin{figure}[h] 
\caption{}\label{LCMHG}
\begin{center}
\begin{tikzpicture}
\shade [shading=ball, ball color=black]  (0,0) circle (.1) node [left] {$f$} node [below] {$5$};
\draw  [shape=circle] (1,-0.5) circle (.1) node [below] {$4$} ;
\draw  [shape=circle] (2,0) circle (.1) node [right]{$3$};
\draw  [shape=circle] (2,1) circle (.1) node [right]{$2$};
\shade [shading=ball, ball color=black](1,1.5) circle (.1) node [left] {$a$} node [above] {$1$};

\draw [line width=1.2pt ] (0,0)--(1,-0.5)  node [pos=.5, below] {$e$};
\draw [line width=1.2pt ] (1,-0.5)--(2,0)  node [pos=.5, below] {$d$};
\draw [line width=1.2pt ] (2,0)--(2,1)  node [pos=.5, right] {$c$};
\draw [line width=1.2pt ] (2,1)--(1,1.5)  node [pos=.5, above] {$b$};
\path (0,0.4)--(2.5,0.4) node [pos=.5, right ] {$g$};
\path [pattern=north west lines, pattern  color=green]   (0,0)--(2,0)--(2,1)--cycle;
\end{tikzpicture}
\end{center}
\end{figure}

\begin{figure}[h] 
\caption{}\label{LCMHG2}
\begin{tikzpicture}[scale=0.32, vertices/.style={draw, circle, inner sep=2pt},every node/.style={scale=0.55}]
              \node [vertices, label=right:{\textcolor{red}{$0$}}] (0) at (-0+0,0){};
              \node [vertices, label=right:{\textcolor{red}{$e f g=\{5\}$}}] (1) at (-14+0,3){};
              \node [vertices, label=right:{\textcolor{red}{$d e=\{4\}$}}] (2) at (-14+7,3){};
              \node [vertices, label=right:{\textcolor{red}{$a b=\{1\}$}}] (12) at (-14+14,3){};
              \node [vertices, label=right:{\textcolor{red}{$c d g=\{3\}$}}] (4) at (-14+21,3){};
              \node [vertices, label=right:{\textcolor{red}{$b c g=\{2\}$}}] (7) at (-14+28,3){};
              \node [vertices, label=right:{\textcolor{red}{$d e f g=\{4,5\}$}}] (3) at (-21+0,6){};
              \node [vertices, label=right:{\textcolor{red}{$a b e f g=\{1,5\}$}}] (13) at (-21+7,6){};
              \node [vertices, label=right:{\textcolor{red}{$a b d e=\{1,4\}$}}] (14) at (-21+14,6){\textcolor{blue}{$g$}};
              \node [vertices, label=right:{\textcolor{red}{$c d e g=\{3,4\}$}}] (5) at (-21+21,6){};
              \node [vertices, label=right:{\textcolor{red}{$b c e f g=\{2,5\}$}}] (8) at (-21+28,6){};
              \node [vertices, label=right:{\textcolor{red}{$a b c g=\{1,2\}$}}] (16) at (-21+35,6){};
              \node [vertices, label=right:{\textcolor{red}{$b c d g=\{2,3\}$}}] (9) at (-21+42,6){};
              \node [vertices, label=right:{\textcolor{red}{$a b d e f g=\{1,4,5\}$}}] (15) at (-14+0,9){\textcolor{blue}{$c$}};
              \node [vertices, label=right:{\textcolor{red}{$c d e f g=\{3,4,5\}$}}] (6) at (-14+7,9){\textcolor{blue}{$b$}};
              \node [vertices, label=right:{\textcolor{red}{$a b c e f g=\{1,2,5\}$}}] (17) at (-14+14,9){\textcolor{blue}{$d$}};
              \node [vertices, label=right:{\textcolor{red}{$a b c d g=\{1,2,3\}$}}] (18) at (-14+21,9){\textcolor{blue}{$e$}};
              \node [vertices, label=right:{\textcolor{red}{$b c d e g=\{2,3,4\}$}}] (10) at (-14+28,9){};
              \node [vertices, label=right:{\textcolor{red}{$b c d e f g=\{2,3,4,5\}$}}] (11) at (-7/2+0,12){\textcolor{blue}{$a$}};
              \node [vertices, label=right:{\textcolor{red}{$a b c d e g=\{1,2,3,4\}$}}] (19) at (-7/2+7,12){\textcolor{blue}{$f$}};
              \node [vertices, label=right:{\textcolor{red}{$a b c d e f g=\{1,2,3,4,5\}$}}] (20) at (-0+0,15){};
      \foreach \to/\from in {0/1, 0/2, 0/12, 0/4, 0/7, 1/8, 1/13, 1/3, 2/5, 2/14, 2/3, 3/6, 3/15, 4/9, 4/5, 5/6, 5/10, 6/11, 7/8, 7/16, 7/9, 8/17, 8/11, 9/18, 9/10, 10/11, 10/19, 11/20, 12/16, 12/13, 12/14, 13/17, 13/15, 14/15, 14/19, 15/20, 16/17, 16/18, 17/20, 18/19, 19/20}
      \draw [-] (\to)--(\from);
      \end{tikzpicture}
      
      \end{figure}

\end{example}

Note that there can be numerous cases where
$\mathcal{H}(I) \not= \mathcal{H}(I')$ but $L_I = L_{I'}$.  In these
cases the difference here between $\mathcal{H}(I)$ and
$\mathcal{H}(I')$ has to be in the edges that do not correspond to
meet-irreducibles. \Cref{removeMeets} determines which  edges in
$\mathcal{H}(I)$ correspond to elements which are not meet-irreducible in the
corresponding $L_I$. We have the following result that is an extension of \Cref{sameBetti}

\begin{corollary}\label{sameTotalBetti}
Let $I_1$ and $I_2$ be square-free monomial ideals such that $\mathcal{H}(I_1)=\mathcal{H}(I_2)\cup F$ where $F\in \mathcal{H}(I_1)$ is  the union
of 2 or more distinct edges of $\mathcal{H}(I_1)$, then the total Betti numbers of two ideals coincide.
\end{corollary}

\begin{proof}
Combining \Cref{removeMeets} and \Cref{ThmLCMHG} with the work of \cite{GPW} and
\cite{mapes}, we can see that removing edges $F$ which are the union of
other edges preserves the LCM-lattice and thus preserves all the total Betti
numbers.
\end{proof}

\begin{example}
In the \Cref{LCMHyper}, the edge corresponding to the variable $g$ is $\{2,3,5\}$ and it is a union of $\{2,3\}$ and $\{5\}$. Hence by \Cref{removeMeets} and \Cref{ThmLCMHG}, the LCM-lattice of the ideal is the same after we remove edge corresponding to the variable $g$. In other word by \Cref{sameTotalBetti}, the projective dimension of the hypergraph is the same as the projective
dimension of the hypergraph after we remove the edge $\{2,3,5\}$. We can see from \Cref{LCMHG2} that $g$ does not correspond to a meet-irreducible of $L_I$.
\end{example}

\begin{example}
In the \Cref{FirstHG}, the edge corresponding to the variable $q$ or
the variable $r$ is a union of 2 or more distinct edges. Hence the
projective dimension of the hypergraph is the same as the projective
dimension of the hypergraph after we remove those two edges.
\end{example}

\begin{remark}\label{explain}
Using \Cref{sameTotalBetti}, we can see that in order to further extend the
previous results on computing the projective dimension by using
combinatorial formulas on $\mathcal{H}(I)$ we need only consider
certain classes of hypergraphs which do not have edges which would be
deemed irrelevant by \Cref{removeMeets}. This explains the results of \cite{KRT} and \cite{LMa1} where they only focus on 1-skeleton of the hypergraph, or more precisely, they focus on the subgraph coming from the open vertices of the 1-skeleton. For example, in the work of \cite{LMa1} , they show that any edges that have closed vertices can be removed without impact the projective dimension (cf. Theorem 2.9 (d) \cite{LMa2}). This is because such kind of edge is a union of other edges, i.e. closed vertices. Moreover, in the work of \cite{KRT}, they only put restrictions on the open vertex subgraph and require the complete bipartite assumption because of the same reason. 
\end{remark}

\section{Bushes with higher dimensional edges}\label{BushesSec}
\begin{definition}\label{bushes}
We say a vertex is a joint on a hypergraph if its degree is at least $3$. We say a hypergraph is a bush, if its 1-skeleton has branches of
length at most 2. 
\end{definition}

The smallest case of a bush is a 2-star where there is exactly one joint
and every branch has length less than or equal 2. 

In this section, we focus on hypergraphs which are bushes and their projective dimension. More precisely, we want to see the impact of the higher dimensional edges on the projective dimension.

One technique that is used in \cite{LMa1} and \cite{LMa2} which we will need
here, is using the short exact sequences obtained by looking at colon
ideals.  Specifically there are two types of colon ideals that we are
interested in, and we explain below what each operation looks like
on the associated hypergraphs.

\begin{definition}
Let $\mathcal{H}$ be a hypergraph, and $I = I(\mathcal{H})$ be the
standard square-free monomial ideal associated to it in the polynomial
ring $R$. Let $\mathcal{G}(I)=\{m_1,\dots, m_{\mu}\}$ be the minimal generating set of $I$. Let $F$ be a
edge in $\mathcal{H}$ and let $x_F \in R$ be the variable associated
to $F$.   Also let
$v$ be a vertex in $\mathcal{H}$ and $m_v \in I$ be the monomial
generator associated to it.
\begin{itemize}
\item The hypergraph $\mathcal{H}_v : v = \mathcal{Q}_v$ is the hypergraph associated to the ideal
$I_v : m_v$ where $I_v =\mathcal{G}(I)\backslash m_{v}$, and
$\mathcal{H}_v = \mathcal{H}(I_v)$ is the hypergraph associated to the ideal $I_v$.  
\item The hypergraph $\mathcal{H} : F$, obtained by removing $F$
  in $\mathcal{H}$, is the hypergraph associated to the ideal $I : x_F$.
  \item The hypergraph $(\mathcal{H},x_F)$, obtained by adding a vertex corresponding to the variable $x_F$ in $\mathcal{H}$, is the hypergraph associated to the ideal $(I,x_F)$.

\end{itemize}
\end{definition}

The following results appearing in \cite{LMa2}
will be very useful to us in this section. We put them here for the self-containment of this work and for the reader's convenience.

\begin{theorem} \label{LaundryList2}
\hangindent\leftmargini
\hskip\labelsep
\begin{enumerate}[label={(\arabic*)},ref={(\arabic*)}]
\item  \label{pdFormulaString} (cf. Corollary 3.8 \cite{LMa1}) An open string hypergraph with $\mu$ vertices has
projective dimension $\mu-\left\lfloor \frac{\mu}{3}\right\rfloor $  
\item  \label{LMa2-2.9c} (cf. Theorem 2.9 (c) \cite{LMa2}) If
  $\mathcal{H}' \subseteq \mathcal{H}$ are hypergraphs with
  $\mu(\mathcal{H}') = \mu(\mathcal{H})$, then $\pd(\mathcal{H}')
  \leq \pd(\mathcal{H})$ where $\mu(\mathcal{*})$ denotes the number of vertices of $*$.
\item \label{LMa2-2.9d} (cf. Theorem 2.9 (d) \cite{LMa2}) Let
$\mathcal{H}', \mathcal{H}$ be hypergraphs with $\mathcal{H} =
\mathcal{H}' \cup F$ where $F = \{i_1,\dots,i_r\}$. If $\{i_j\} \in
\mathcal{H}'$ for all $j$, then $\pd(\mathcal{H}') =
\pd(\mathcal{H}:F) = \pd (\mathcal{H})$.
\item \label{LMa2-4.7,4.9} (cf. Proposition 4.7 and 4.9 \cite{LMa2})
  Let $\mathcal{H}$ be a 1-dimensional hypergraph, $w$ a vertex with
  degree at least $3$ in $\mathcal{H}$, and $S$ be a branch departing
  from $w$ with $\mu$ vertices. Suppose all the vertices of $S$ are open except the end vertex, and let $E$ be the edge connecting $w$ to $S$. Then $\pd(\mathcal{H}) = \pd(\mathcal{H}')$, where $\mathcal{H}'$ is the following hypergraph: (a) if $n\equiv1$ mod $3$,then $\mathcal{H}' =\mathcal{H}:E$; (b) if $n\equiv2$ mod $3$, then $\mathcal{H}'=\mathcal{H}_w$
\end{enumerate}
\end{theorem}

The following lemma extends the result of  \cite{LMa2}, because of the new connection of LCM-lattics and hypergraphs. It is a special case in the work of \cite{KRT}. We provide a different proof here. 

\begin{lemma}\label{2Stars}Let $\mathcal{H}$ be a hypergraph such that its 1-skeleton is a 2-star. Then
$\pd(\mathcal{H})=|V(\mathcal{H})|-1$.
\end{lemma}

\begin{proof}
Suppose there is no higher dimensional edges on $\mathcal{H}$ then it is true
by Proposition 4.16 of \cite{LMa2}. Assume there is an edge $F$ on $\mathcal{H}$ such that
it is at least 2-dimensional. We assume that $F\neq\cup F_{i}$ otherwise
we are done because of \Cref{removeMeets}. We use induction on the number of edges on $\mathcal{H}$ such
that their dimension is at least 2.

Suppose $F$ is the only higher edge on $\mathcal{H}$ such that it
has dimension at least 2. Notice that the number of vertices of
$F$ must be at least 3. Therefore the number of vertices of $\mathcal{H}_{V(F)}$
is at most $|V(\mathcal{H})|-3$ and the projective dimension of $\mathcal{H}_{V(F)}$
is at most $|V(\mathcal{H})|-3$. We consider the short exact sequence

\[0\leftarrow (\mathcal{H},x_{F})\leftarrow\mathcal{H}\leftarrow (\mathcal{H}:F) \leftarrow0\]
where $x_{F}$ is the variable corresponding to the edge $F$. The
hypergraph $\mathcal{H}:F$ is a 2-star without higher dimensional
edge with the same vertices of $\mathcal{H}$ and hence $\text{pd}(\mathcal{H}:F)=|V(\mathcal{H})|-1$.
Moreover, by \Cref{LaundryList2}\ref{LMa2-2.9c}, we have $\text{pd}(\mathcal{H}:F)\leq\text{pd}(\mathcal{H})$.
We observe that the hypergraph $(\mathcal{H},x_{F})$ is the union
of $\mathcal{H}_{V(F)}$ and an isolated vertex corresponding to $x_{F}$,
hence it has projective dimension at most $|V(\mathcal{H})|-3+1=|V(\mathcal{H})|-2$.
Using the short exact sequence on the projective dimension, we have
\[
\text{pd}(\mathcal{H})\leq\max\{\text{pd}(\mathcal{H}:F),\text{pd}(\mathcal{H},x_{F})\}=\text{pd}(\mathcal{H}:F)\leq\text{pd}(\mathcal{H}).
\]

Now we assume that $\mathcal{H}$ has more than one higher dimensional
edge and $F$ is one of them. The induction hypothesis gives $\text{pd}(\mathcal{H}:F)=|V(\mathcal{H}:F)|-1=|V(\mathcal{H})|-1$
since $(\mathcal{H}:F)$ has the same number of vertices of $\mathcal{H}$
with one less higher dimensional edge $F$. As before, we use the
same short exact sequence above and the fact that $\text{pd}(\mathcal{H},x_{F})\leq|V(\mathcal{H})|-2$
to obtain $\text{pd}(\mathcal{H})=\text{pd}(\mathcal{H}:F)=|V(\mathcal{H})|-1$. 
\end{proof}

The following lemma is a rewrite of the result in \cite{LMa2} where they did not see the more broader implication.

\begin{lemma}\label{RemoveJoint}Let $J$ be a joint on a hypergraph $\mathcal{H}$ such that
there is no higher dimensional edge on the branches of $J$ and there
is a branch on $J$ with length 2. Then $\pd(\mathcal{H})=\pd(\mathcal{H}')$
where $\mathcal{H}'$ is obtained by removing $J$ from $\mathcal{H}$. 
\end{lemma}

\begin{proof} The proof follows exclusively as the proof of Proposition 4.9 of \cite{LMa2} or \Cref{LaundryList2}\ref{LMa2-4.7,4.9}. The only assumption that is needed in the proof of  Proposition 4.9 of \cite{LMa2} is that $J$ has no higher dimensional edge on the branches of $J$.
\end{proof}

We conclude this section with the following proposition which is the extension of results of \cite{LMa2}. Notice that the connected closed vertices assumption is harmless via \Cref{LaundryList2} \ref{LMa2-2.9d}. This also shows that higher dimensional edges can be removed or disregarded with the new connection we built from previous sections.

\begin{proposition}
\label{Remove2Star}Let $\mathcal{H}$ be a bush and it does not have connected
closed vertices. If all the higher dimensional edges of $\mathcal{H}$
have vertices on the branches of the same joint, then $\pd(\mathcal{H})=\pd(\mathcal{H}')$
where $\mathcal{H}'$ is obtained by removing all of joints of $\mathcal{H}$
having branches of length 2.
\end{proposition}

\begin{proof}
We use induction on the number of joints and number of higher dimensional
edges on the joints. Suppose $\mathcal{H}$ only has one joint and
this joint has branches length 1, then nothing to be proven. Suppose that $\mathcal{H}$ has an unique joint with at least one branch of
length 2. Then by \Cref{2Stars}, we are done.

Suppose $\mathcal{H}$ has at least two joints, and we assume that
$J$ is a joint having branches of length 2. Suppose there is
no higher dimensional edges on branches of $J$ then by \Cref{RemoveJoint}, $\pd(\mathcal{H})=\pd(\mathcal{H}_{J})$
where $\mathcal{H}_{J}$ is the hypergraph obtained from $\mathcal{H}$
by removing the joint $J$. Notice that $\mathcal{H}_{J}$
is a unions of branches of $J$ and another hypergraph satisfies
the assumptions of the theorem. Moreover, the vertices of branches
of $J$ all become closed in $\mathcal{H}_{J}$ because the
length of branches are at most 2. Let $V(J)$ be the vertex set
of all branches of $J$ including $J$ and $\mathcal{H}_{V(J)}$
be the hypergraph obtained from $\mathcal{H}$ by removing all the
vertices of $V(J)$ and $\mathcal{H}_{V(J)}'$ be the hypergraph
obtained from $\mathcal{H}_{V(J)}$ by removing all of the joints
of $\mathcal{H}_{V(J)}$ having branches of length 2. Then $\pd(\mathcal{H}_{J})=|V(J)|-1+\pd(\mathcal{H}_{V(J)})=|V(J)|-1+\pd(\mathcal{H}_{V(J)}')$
by induction. Since $\mathcal{H}'$ is a union of branches of $J$
without $J$ and $\mathcal{H}_{V(J)}'$, we have $\pd(\mathcal{H}')=\pd(\mathcal{H}_{J})=\pd(\mathcal{H})$.

Suppose the branches of $J$ has at least one higher dimensional
edge. Let $F$ be one of higher dimensional edge and $x_F$ be the variable corresponding to the edge. We consider the same short exact
sequence:\[
0\leftarrow (\mathcal{H},x_F)\leftarrow\mathcal{H}\leftarrow \mathcal{H}:F \leftarrow0.
\]
Notice that $(\mathcal{H}:F)$ is a hypergraph obtained from $\mathcal{H}$ with the edge $F$ removed. By induction hypothesis, $\pd(\mathcal{H}:F)=\pd((\mathcal{H}:F)')$
where $(\mathcal{H}:F)'$ is obtained from $\mathcal{H}:F$ by removing
joints having branches of length 2. Since $J$ is a joint with
branches of length 2, $J$ will be removed in $\mathcal{H}'$
and $(\mathcal{H}:F)'$. Moreover, all the vertices of branches of
$J$ will become closed because the branches have length at most
2. By \Cref{LaundryList2}\ref{LMa2-2.9d} again, we have $\pd(\mathcal{H}:F)'=\pd(\mathcal{H}')$.
We are left to show $\pd(\mathcal{H}:F)=\pd(\mathcal{H})$
as before. With the short exact sequence it is sufficient to show
that $\pd(\mathcal{H}:F)>\pd(\mathcal{H},x_F)$. 

By induction on the number of higher dimensional edges on the branches
of $J$, we have $\pd(\mathcal{H},x_F)=\pd((\mathcal{H},x_F)')$ because $(\mathcal{H},x_F)$ has no edge $F$ on the branches of $J$.
Once we show $\pd((\mathcal{H}:F)')>\pd((\mathcal{H},x_F)')$
then we are done. The hypergraphs $(\mathcal{H}:F)'$
and $(\mathcal{H},x_F)'$ are unions of branches of $J$ and the
hypergraph $(\mathcal{H}:F)'_{V(J)}=(\mathcal{H},x_F)'_{V(J)}$.  We now just need to compare
the structure of branches of $J$ on $(\mathcal{H}:F)'$
and $(\mathcal{H},x_F)'$.
The vertices of branches of $J$ on
$(\mathcal{H}:F)'$ and $(\mathcal{H},x_F)'$ are closed because $J$
is a joint with branches of length 2. Hence the projective
dimension of the branches of $J$ is the number of vertices. The number of vertices on the branches of $J$ in $(\mathcal{H}:F)'$
is $|V(J)|-1$ but the number of vertices on the branches of $J$
in $(\mathcal{H},x_F)'$ is $|V(J)|-|V(F)|+1<|V(J)|-1$. Hence
we have $\pd(\mathcal{H}:F)=\pd((\mathcal{H}:F)')>\pd((\mathcal{H},x_F)')=\pd(\mathcal{H},x_F)$.
\end{proof}

\Cref{Remove2Star} provides us an inducting process to obtain the projective dimension of a hypergraph or a square-free monomial. One just needs to remove joints one by one until all of the branches of length 2 are seperated. This actually covers a large class of ideals. In the next section, we provide the process to see the efficient reduction.

\section{Appendix: algorithmic procedures and one example}

We say a hypergraph $\mathcal{H}$ with $V=[\mu]$ is a
\emph{string} if $\{i, i+1\}$ is in $\mathcal{H}$ for all $i = 1,
\dots, \mu-1$, and the only edges containing $i$ are $\{i-1, i\}$,
$\{i, i+1\}$ and possibly $\{i\}$.  Also, $\mathcal{H}$ is a \emph{$\mu$-cycle} if $\mathcal{H} = \tilde{\mathcal{H}}
  \cup \{\mu,1\}$ where $\tilde{\mathcal{H}}$ is a string. Let $\mathcal{H}$ be a hypergraph satisfies the assumptions of \Cref{Remove2Star}. Further more, if $\mathcal{H}'$ is a union of bushes, 2-star strings, and cycles of 2-stars,
then one can obtain $\pd(\mathcal{H})$ by first removing all
the joints having branches of length 2, and then one can apply the Proposition 4.18 in \cite{LMa2} to obtain the projective dimension
of $\mathcal{H}$. This is because all the higher dimensional edges
on the branches of joints of length 2 can be removed in $\mathcal{H}'$
by \Cref{LaundryList2}\ref{LMa2-2.9d} and the fact that all the
vertices on the branches of joints of length 2 become closed in $\mathcal{H}'$. In this section, we present algorithmic procedures to compute the projective dimension of a bush hypergraph. We use Algorithm A.1 in \cite{LMa2} to decide if a vertex is a joint or an endpoint.  We write $d(i)$ as the degree of any given vertex $i$ in $\mathcal{H}$. We just need to know if $d(i)=0$, $1$, $2$ or greater than $2$ (a joint) for the purpose of computation.

The following result provides an algorithm to obtain the hypergraph $\mathcal{H'}$ in the statement of \Cref{Remove2Star}. We use the variable $i$ to detect the vertices with degree one (if any). The variable $j$ runs through the other vertices looking for neighbors of $i$, and $k$ looks for the other neighbor of $j$ (if any).

\begin{algorithm}\label{removejoint2}
Input: A connected hypergraph $\mathcal{H}$ that is a bush and all the higher dimensional edges of $\mathcal{H}$
have vertices on the branches of the same joint.  Let the vertex set be $V(\mathcal{H})=\{1,2,\cdots,\mu\}$. The output is: a hypergraph $\mathcal{H'}$ such that $\pd(\mathcal{H})=\pd(\mathcal{H'})$.

\begin{itemize}

\item [Step 0:] Set $i=1$.

\item [Step 1:] If $\mathcal{H}=\emptyset$, stop the process and give $\mathcal{H}$ as output.\\
If $\mathcal{H}\neq \emptyset$ set $j=k=1$, and do the following: if $i\leq \mu$,  then go to Step 2, if $i=\mu+1$, then stop the process and give $\mathcal{H}$ as output.
\item [Step 2:] If $i\notin V(\mathcal{H})$, then set $i=i+1$ and start Step 2 again.\\
If $i\in V(\mathcal{H})$, compute $d(i)$ using Algorithm A.1 in \cite{LMa2}\\
\indent if $d(i)=0,1,2$, then set $i=i+1$ and start Step 2 again;\\
\indent if $d(i)>2$ then go to Step 3;

\item [Step 3:] If $j>\mu$ then set $i=i+1$ and go to Step 2. If $j=i$ or if $\{i,j\}\notin\mathcal{H}$ then set $j=j+1$ and start again Step 3. If $\{i,j\}\in\mathcal{H}$ then go to Step 4. 

\item [Step 4:] Use Algorithm A.1 in \cite{LMa2} to compute $d(j)$. 
 If $d(j)=2$, then go to Step 5; otherwise set $j=j+1$ and go to Step 3.

\item [Step 5:] If $k=i$, or if $k=j$, or if $\{j,k\}\notin\mathcal{H}$, then set $k=k+1$ and start again Step 5. If $\{j,k\}\in\mathcal{H}$ then set $\mathcal{H}=\mathcal{H}_{i}$ and $i=i+1$, go to Step 1. (this procedure stops because $d(j)=2$)
 
\end{itemize}
\end{algorithm}

The following result provides an algorithm to remove higher dimensional edges of a hypergraph such that the projective dimension stays the same.

\begin{algorithm}\label{removeHigheredge}
Input: A connected hypergraph $\mathcal{H}=\cup \{F_i\}_{i=1}^p$ such that all edges that have cardinality greater than 2 are union of 2 or more distinct edges of $\mathcal{H}$.  The output is: a hypergraph $\mathcal{H'}$ such that $\pd(\mathcal{H})=\pd(\mathcal{H'})$.
\begin{itemize}

\item [Step 0:] Set $i=1$.

\item [Step 1:] If $\mathcal{H}=\emptyset$, stop the process and give $\mathcal{H}$ as output.\\
If $\mathcal{H}\neq \emptyset$ then do the following: if $i\leq p$,  then go to Step 2, if $i=p+1$, then  stop the process and give $\mathcal{H}$ as output.

\item [Step 2:] If $|F_i|<3$ then set $i=i+1$ and start Step 1 again. If $|F_i|\geq 3$ then set $\mathcal{H}=\mathcal{H}\backslash \{F_i\}$ and $i=i+1$, and go to Step 1. 
\end{itemize}
\end{algorithm}

\begin{remark} 
We can combine Algorithm 5.6 in \cite{LMa1}, Algorithms A.1 and A.2 in \cite{LMa2}, and \Cref{removejoint2} and \ref{removeHigheredge} to compute the projective dimension of a hypergraph that is a bush and its higher dimensional edges are on the same joints. The example below illustrates the process.  
\end{remark}

\begin{example}
Let $\mathcal{H}$ be a hypergraph as in \Cref{Brushes}. By
\Cref{Remove2Star} and \Cref{removejoint2}, we remove the red vertices that are the joints of $\mathcal{H}$ having branches of length 2 to obtain the hypergraph as in \Cref{RemoveJointF}. By \Cref{removeMeets} and \Cref{removeHigheredge}, we can remove the higher dimension green edges and by \Cref{LaundryList2}\ref{LMa2-2.9d} and \Cref{removeHigheredge} again, we can remove the blue edges. We obtain the hypergraph as in \Cref{RemoveJointedge}. Finally, we remove edges using \Cref{LaundryList2}\ref{LMa2-2.9d}, \Cref{LaundryList2}\ref{LMa2-4.7,4.9}, and Algorithms A.2 in \cite{LMa2} to obtain the hypergraph as in \Cref{RemoveJointedgeRduce}. Then by \Cref{LaundryList2}\ref{pdFormulaString} and  Algorithm 5.1 in \cite{LMa1}, we have the project dimension of $\mathcal{H}$ equal to $27+2+2+4=35$ which is coming from 27 isolated vertices, two open strings of length 3, and a string of length 5. 
\begin{figure}[h] 
\caption{}\label{Brushes}
\begin{center}
\begin{tikzpicture}[scale=0.85]
\shade [shading=ball, ball color=black] (-2,1) circle (.15);
\shade [shading=ball, ball color=black] (-2,0) circle (.15);
\shade [shading=ball, ball color=black] (0.5,-1) circle (.15);
\shade [shading=ball, ball color=black] (0,-1) circle (.15);
\draw  [shape=circle] (-1,1) circle (.15);
\shade [shading=ball, ball color=black] (-1,0) circle (.15);
\shade [shading=ball, ball color=black] (0.5,0.5) circle (.15);
\shade [shading=ball, ball color=black] (1.5,0.5) circle (.15);
\shade [shading=ball, ball color=black] (1,0.5) circle (.15);
\draw  [shape=circle] (-1.5,0.5) circle (.15);
\draw  [shape=circle] (0,0) circle (.15);
\draw  [shape=circle] (0,1) circle (.15);
\draw  [shape=circle] (1,0) circle (.15);
\draw  [shape=circle, color=red] (2,0) circle (.15);
\shade [shading=ball, ball color=red]  (3,0) circle (.15);
\draw  [shape=circle] (4,0) circle (.15);
\draw  [shape=circle] (5,0) circle (.15);
\draw  [shape=circle] (6,0) circle (.15);
\shade [shading=ball, ball color=black]  (7,0) circle (.15);

\shade [shading=ball, ball color=black] (1.5,-1) circle (.15);
\draw  [shape=circle] (2,-1) circle (.15);
\shade [shading=ball, ball color=black] (2,-1.5) circle (.15);

\shade [shading=ball, ball color=black]  (2.5,-1) circle (.15);
\draw  [shape=circle] (3.5,-1) circle (.15);
\shade [shading=ball, ball color=black]  (4.5,-1) circle (.15);

\draw  [shape=circle] (1,1) circle (.15);
\shade [shading=ball, ball color=red] (2,1) circle (.15);
\draw  [shape=circle] (3,1) circle (.15);
\draw  [shape=circle] (4,1) circle (.15);
\shade [shading=ball, ball color=red]  (5,1) circle (.15);
\draw  [shape=circle] (6,1) circle (.15);
\shade [shading=ball, ball color=black]  (7,1) circle (.15);
\draw  [shape=circle] (1.5,1.5) circle (.15);
\shade [shading=ball, ball color=black] (0.5,1.5) circle (.15);
\draw  [shape=circle] (6,1.5) circle (.15);
\shade [shading=ball, ball color=black]  (7,1.5) circle (.15);
\draw  [shape=circle] (2.5,1.5) circle (.15);
\draw [shape=circle, color=red]  (3,1.5) circle (.15);
\draw [shape=circle]  (4,1.5) circle (.15);
\shade [shading=ball, ball color=black]   (5,1.5) circle (.15);
\draw [shape=circle]  (3,2.5) circle (.15);
\shade [shading=ball, ball color=black]   (4,2.5) circle (.15);
\shade [shading=ball, ball color=black]  (6.5,-1) circle (.15);

\path [pattern=north west lines, pattern  color=green]   (-1,1)--(0,1)--(0,0)--(-1,0)--cycle;
\path [pattern=north west lines, pattern  color=green]   (2,1)--(2.5,1.5)--(1.5,1.5)--cycle;
\path [pattern=north west lines, pattern  color=blue]   (4,2.5)--(3,2.5)--(4,1.5)--cycle;
\path [pattern=north west lines, pattern  color=blue]   (6,1.5)--(6,0)--(7,1.5)--cycle;
\path [pattern=north west lines, pattern  color=blue]   (3,1)--(4,1)--(4,0)--(3.5,-1)--cycle;

\draw [line width=1.2pt ] (6,1.5)--(7,1.5)  ;
\draw [line width=1.2pt ] (6,1.5)--(5,1)  ;
\draw [line width=1.2pt ] (3,1.5)--(3,2.5)  ;
\draw [line width=1.2pt ] (3,1.5)--(5,1.5)  ;
\draw [line width=1.2pt ] (3,1)--(3,0)  ;
\draw [line width=1.2pt ] (4,2.5)--(3,2.5)  ;

\draw [line width=1.2pt ] (-1.5,0.5)--(-2,0)  ;
\draw [line width=1.2pt ] (-1.5,0.5)--(-2,1)  ;
\draw [line width=1.2pt ] (0,0)--(0.5,-1)  ;
\draw [line width=1.2pt ] (0,0)--(0,-1)  ;
\draw [line width=1.2pt ] (-1.5,0.5)--(-1,0)  ;
\draw [line width=1.2pt ] (-1.5,0.5)--(-1,1)  ;
\draw [line width=1.2pt ] (0,0)--(-1,0)  ;
\draw [line width=1.2pt ] (0,1)--(-1,1)  ;
\draw [line width=1.2pt ] (0,1)--(1,1)  ;
\draw [line width=1.2pt ] (0,0)--(1,0)  ;
\draw [line width=1.2pt ] (1,0)--(2,0)  ;
\draw [line width=1.2pt ] (2,0)--(3,0)  ;
\draw [line width=1.2pt ] (4,0)--(3,0)  ;
\draw [line width=1.2pt ] (4,0)--(5,0)  ;
\draw [line width=1.2pt ] (6,0)--(7,0)  ;
\draw [line width=1.2pt ] (6,0)--(5,1)  ;

\draw [line width=1.2pt ] (7,1)--(6,1)  ;
\draw [line width=1.2pt ] (5,1)--(6,1)  ;
\draw [line width=1.2pt ] (1,1)--(.5,0.5)  ;
\draw [line width=1.2pt ] (1,1)--(1.5,0.5)  ;
\draw [line width=1.2pt ] (1,1)--(1,0.5)  ;
\draw [line width=1.2pt ] (5,1)--(5,0)  ;
\draw [line width=1.2pt ] (5,1)--(4,1)  ;

\draw [line width=1.2pt ] (3,1)--(4,1)  ;
\draw [line width=1.2pt ] (2,1)--(3,1)  ;
\draw [line width=1.2pt ] (1,1)--(2,1)  ;

\draw [line width=1.2pt ] (2,1)--(1.5,1.5)  ;
\draw [line width=1.2pt ] (0.5,1.5)--(1.5,1.5)  ;

\draw [line width=1.2pt ] (2,1)--(2.5,1.5)  ;
\draw [line width=1.2pt ] (3,1.5)--(2.5,1.5)  ;
\draw [line width=1.2pt ] (2,-1)--(2,0)  ;
\draw [line width=1.2pt ] (2.5,-1)--(2,0)  ;
\draw [line width=1.2pt ] (1.5,-1)--(2,0)  ;
\draw [line width=1.2pt ] (2,-1)--(2,-1.5)  ;

\draw [line width=1.2pt] (3.5,-1)--(3,0)  ;
\draw [line width=1.2pt ] (3.5,-1)--(4.5,-1)  ;

\draw [line width=1.2pt] (6.5,-1)--(5,0)  ;
\end{tikzpicture}
\end{center}
\end{figure}
\begin{figure}[h] 
\caption{}\label{RemoveJointF}
\begin{center}
\begin{tikzpicture}[scale=0.85]
\shade [shading=ball, ball color=black] (-2,1) circle (.15);
\shade [shading=ball, ball color=black] (-2,0) circle (.15);
\shade [shading=ball, ball color=black] (0.5,-1) circle (.15);
\shade [shading=ball, ball color=black] (0,-1) circle (.15);
\draw  [shape=circle] (-1,1) circle (.15);
\shade [shading=ball, ball color=black] (-1,0) circle (.15);
\shade [shading=ball, ball color=black] (0.5,0.5) circle (.15);
\shade [shading=ball, ball color=black] (1.5,0.5) circle (.15);
\shade [shading=ball, ball color=black] (1,0.5) circle (.15);
\draw  [shape=circle] (-1.5,0.5) circle (.15);
\draw  [shape=circle] (0,0) circle (.15);
\draw  [shape=circle] (0,1) circle (.15);
\shade [shading=ball, ball color=black] (1,0) circle (.15);

\shade [shading=ball, ball color=black]  (4,0) circle (.15);
\draw  [shape=circle] (5,0) circle (.15);
\shade [shading=ball, ball color=black]  (6,0) circle (.15);
\shade [shading=ball, ball color=black]  (7,0) circle (.15);

\shade [shading=ball, ball color=black] (1.5,-1) circle (.15);
\shade [shading=ball, ball color=black]  (2,-1) circle (.15);
\shade [shading=ball, ball color=black] (2,-1.5) circle (.15);

\shade [shading=ball, ball color=black]  (2.5,-1) circle (.15);
\shade [shading=ball, ball color=black] (3.5,-1) circle (.15);
\shade [shading=ball, ball color=black]  (4.5,-1) circle (.15);

\draw  [shape=circle] (1,1) circle (.15);

\shade [shading=ball, ball color=black] (3,1) circle (.15);
\shade [shading=ball, ball color=black] (4,1) circle (.15);

\shade [shading=ball, ball color=black]  (6,1) circle (.15);
\shade [shading=ball, ball color=black]  (7,1) circle (.15);
\shade [shading=ball, ball color=black] (1.5,1.5) circle (.15);
\shade [shading=ball, ball color=black] (0.5,1.5) circle (.15);
\shade [shading=ball, ball color=black]  (6,1.5) circle (.15);
\shade [shading=ball, ball color=black]  (7,1.5) circle (.15);
\shade [shading=ball, ball color=black]  (2.5,1.5) circle (.15);

\shade [shading=ball, ball color=black] (4,1.5) circle (.15);
\shade [shading=ball, ball color=black]   (5,1.5) circle (.15);
\shade [shading=ball, ball color=black] (3,2.5) circle (.15);
\shade [shading=ball, ball color=black]   (4,2.5) circle (.15);
\shade [shading=ball, ball color=black]  (6.5,-1) circle (.15);

\path [pattern=north west lines, pattern  color=green]   (-1,1)--(0,1)--(0,0)--(-1,0)--cycle;
\draw [line width=1.2pt ] (2.5,1.5)--(1.5,1.5);
\path [pattern=north west lines, pattern  color=blue]   (4,2.5)--(3,2.5)--(4,1.5)--cycle;
\path [pattern=north west lines, pattern  color=blue]   (6,1.5)--(6,0)--(7,1.5)--cycle;
\path [pattern=north west lines, pattern  color=blue]   (3,1)--(4,1)--(4,0)--(3.5,-1)--cycle;

\draw [line width=1.2pt ] (6,1.5)--(7,1.5)  ;
\draw [line width=1.2pt ] (5,1.5)--(4,1.5)  ;
\draw [line width=1.2pt ] (4,2.5)--(3,2.5)  ;

\draw [line width=1.2pt ] (-1.5,0.5)--(-2,0)  ;
\draw [line width=1.2pt ] (-1.5,0.5)--(-2,1)  ;
\draw [line width=1.2pt ] (0,0)--(0.5,-1)  ;
\draw [line width=1.2pt ] (0,0)--(0,-1)  ;
\draw [line width=1.2pt ] (-1.5,0.5)--(-1,0)  ;
\draw [line width=1.2pt ] (-1.5,0.5)--(-1,1)  ;
\draw [line width=1.2pt ] (0,0)--(-1,0)  ;
\draw [line width=1.2pt ] (0,1)--(-1,1)  ;
\draw [line width=1.2pt ] (0,1)--(1,1)  ;
\draw [line width=1.2pt ] (0,0)--(1,0)  ;

\draw [line width=1.2pt ] (4,0)--(5,0)  ;
\draw [line width=1.2pt ] (6,0)--(7,0)  ;

\draw [line width=1.2pt ] (7,1)--(6,1)  ;

\draw [line width=1.2pt ] (1,1)--(.5,0.5)  ;
\draw [line width=1.2pt ] (1,1)--(1.5,0.5)  ;
\draw [line width=1.2pt ] (1,1)--(1,0.5)  ;

\draw [line width=1.2pt ] (3,1)--(4,1)  ;

\draw [line width=1.2pt ] (0.5,1.5)--(1.5,1.5)  ;

\draw [line width=1.2pt ] (2,-1)--(2,-1.5)  ;

\draw [line width=1.2pt ] (3.5,-1)--(4.5,-1)  ;

\draw [line width=1.2pt] (6.5,-1)--(5,0)  ;

\end{tikzpicture}
\end{center}
\end{figure}
\begin{figure}[h] 
\caption{}\label{RemoveJointedge}
\begin{center}
\begin{tikzpicture}[scale=0.85]
\shade [shading=ball, ball color=black] (-2,1) circle (.15);
\shade [shading=ball, ball color=black] (-2,0) circle (.15);
\shade [shading=ball, ball color=black] (0.5,-1) circle (.15);
\shade [shading=ball, ball color=black] (0,-1) circle (.15);
\draw  [shape=circle] (-1,1) circle (.15);
\shade [shading=ball, ball color=black] (-1,0) circle (.15);
\shade [shading=ball, ball color=black] (0.5,0.5) circle (.15);
\shade [shading=ball, ball color=black] (1.5,0.5) circle (.15);
\shade [shading=ball, ball color=black] (1,0.5) circle (.15);
\draw  [shape=circle] (-1.5,0.5) circle (.15);
\draw  [shape=circle] (0,0) circle (.15);
\draw  [shape=circle] (0,1) circle (.15);
\shade [shading=ball, ball color=black] (1,0) circle (.15);

\shade [shading=ball, ball color=black]  (4,0) circle (.15);
\draw  [shape=circle] (5,0) circle (.15);
\shade [shading=ball, ball color=black]  (6,0) circle (.15);
\shade [shading=ball, ball color=black]  (7,0) circle (.15);

\shade [shading=ball, ball color=black] (1.5,-1) circle (.15);
\shade [shading=ball, ball color=black]  (2,-1) circle (.15);
\shade [shading=ball, ball color=black] (2,-1.5) circle (.15);

\shade [shading=ball, ball color=black]  (2.5,-1) circle (.15);
\shade [shading=ball, ball color=black] (3.5,-1) circle (.15);
\shade [shading=ball, ball color=black]  (4.5,-1) circle (.15);

\draw  [shape=circle] (1,1) circle (.15);

\shade [shading=ball, ball color=black] (3,1) circle (.15);
\shade [shading=ball, ball color=black] (4,1) circle (.15);

\shade [shading=ball, ball color=black]  (6,1) circle (.15);
\shade [shading=ball, ball color=black]  (7,1) circle (.15);
\shade [shading=ball, ball color=black] (1.5,1.5) circle (.15);
\shade [shading=ball, ball color=black] (0.5,1.5) circle (.15);
\shade [shading=ball, ball color=black]  (6,1.5) circle (.15);
\shade [shading=ball, ball color=black]  (7,1.5) circle (.15);
\shade [shading=ball, ball color=black]  (2.5,1.5) circle (.15);

\shade [shading=ball, ball color=black] (4,1.5) circle (.15);
\shade [shading=ball, ball color=black]   (5,1.5) circle (.15);
\shade [shading=ball, ball color=black] (3,2.5) circle (.15);
\shade [shading=ball, ball color=black]   (4,2.5) circle (.15);
\shade [shading=ball, ball color=black]  (6.5,-1) circle (.15);
\draw [line width=1.2pt ] (2.5,1.5)--(1.5,1.5);

\draw [line width=1.2pt ] (6,1.5)--(7,1.5)  ;
\draw [line width=1.2pt ] (5,1.5)--(4,1.5)  ;
\draw [line width=1.2pt ] (4,2.5)--(3,2.5)  ;
\draw [line width=1.2pt ] (-1.5,0.5)--(-2,0)  ;
\draw [line width=1.2pt ] (-1.5,0.5)--(-2,1)  ;
\draw [line width=1.2pt ] (0,0)--(0.5,-1)  ;
\draw [line width=1.2pt ] (0,0)--(0,-1)  ;
\draw [line width=1.2pt ] (-1.5,0.5)--(-1,0)  ;
\draw [line width=1.2pt ] (-1.5,0.5)--(-1,1)  ;
\draw [line width=1.2pt ] (0,0)--(-1,0)  ;
\draw [line width=1.2pt ] (0,1)--(-1,1)  ;
\draw [line width=1.2pt ] (0,1)--(1,1)  ;
\draw [line width=1.2pt ] (0,0)--(1,0)  ;

\draw [line width=1.2pt ] (4,0)--(5,0)  ;
\draw [line width=1.2pt ] (6,0)--(7,0)  ;
\draw [line width=1.2pt ] (7,1)--(6,1)  ;

\draw [line width=1.2pt ] (1,1)--(.5,0.5)  ;
\draw [line width=1.2pt ] (1,1)--(1.5,0.5)  ;
\draw [line width=1.2pt ] (1,1)--(1,0.5)  ;
\draw [line width=1.2pt ] (3,1)--(4,1)  ;
\draw [line width=1.2pt ] (0.5,1.5)--(1.5,1.5)  ;
\draw [line width=1.2pt ] (2,-1)--(2,-1.5)  ;

\draw [line width=1.2pt ] (3.5,-1)--(4.5,-1)  ;

\draw [line width=1.2pt] (6.5,-1)--(5,0)  ;

\end{tikzpicture}
\end{center}
\end{figure}

\begin{figure}[h] 
\caption{}\label{RemoveJointedgeRduce}
\begin{center}
\begin{tikzpicture}[scale=0.85]
\shade [shading=ball, ball color=black] (-2,1) circle (.15);
\shade [shading=ball, ball color=black] (-2,0) circle (.15);
\shade [shading=ball, ball color=black] (0.5,-1) circle (.15);
\shade [shading=ball, ball color=black] (0,-1) circle (.15);
\shade [shading=ball, ball color=black] (-1,1) circle (.15);
\shade [shading=ball, ball color=black] (-1,0) circle (.15);
\shade [shading=ball, ball color=black] (0.5,0.5) circle (.15);
\shade [shading=ball, ball color=black] (1.5,0.5) circle (.15);
\shade [shading=ball, ball color=black] (1,0.5) circle (.15);
\draw  [shape=circle] (-1.5,0.5) circle (.15);
\draw  [shape=circle] (0,0) circle (.15);
\shade [shading=ball, ball color=black]  (0,1) circle (.15);
\shade [shading=ball, ball color=black] (1,0) circle (.15);

\shade [shading=ball, ball color=black]  (4,0) circle (.15);
\draw  [shape=circle] (5,0) circle (.15);
\shade [shading=ball, ball color=black]  (6,0) circle (.15);
\shade [shading=ball, ball color=black]  (7,0) circle (.15);

\shade [shading=ball, ball color=black] (1.5,-1) circle (.15);
\shade [shading=ball, ball color=black]  (2,-1) circle (.15);
\shade [shading=ball, ball color=black] (2,-1.5) circle (.15);

\shade [shading=ball, ball color=black]  (2.5,-1) circle (.15);
\shade [shading=ball, ball color=black] (3.5,-1) circle (.15);
\shade [shading=ball, ball color=black]  (4.5,-1) circle (.15);

\draw  [shape=circle] (1,1) circle (.15);

\shade [shading=ball, ball color=black] (3,1) circle (.15);
\shade [shading=ball, ball color=black] (4,1) circle (.15);

\shade [shading=ball, ball color=black]  (6,1) circle (.15);
\shade [shading=ball, ball color=black]  (7,1) circle (.15);
\shade [shading=ball, ball color=black] (1.5,1.5) circle (.15);
\shade [shading=ball, ball color=black] (0.5,1.5) circle (.15);
\shade [shading=ball, ball color=black]  (6,1.5) circle (.15);
\shade [shading=ball, ball color=black]  (7,1.5) circle (.15);
\shade [shading=ball, ball color=black]  (2.5,1.5) circle (.15);

\shade [shading=ball, ball color=black] (4,1.5) circle (.15);
\shade [shading=ball, ball color=black]   (5,1.5) circle (.15);
\shade [shading=ball, ball color=black] (3,2.5) circle (.15);
\shade [shading=ball, ball color=black]   (4,2.5) circle (.15);
\shade [shading=ball, ball color=black]  (6.5,-1) circle (.15);
\draw [line width=1.2pt ] (-1.5,0.5)--(-1,0)  ;
\draw [line width=1.2pt ] (-1.5,0.5)--(-1,1)  ;
\draw [line width=1.2pt ] (0,0)--(-1,0)  ;
\draw [line width=1.2pt ] (0,1)--(1,1)  ;
\draw [line width=1.2pt ] (0,0)--(1,0)  ;
\draw [line width=1.2pt ] (4,0)--(5,0)  ;
\draw [line width=1.2pt] (6.5,-1)--(5,0)  ;
\draw [line width=1.2pt ] (1,1)--(1.5,0.5)  ;
\draw [line width=1.2pt ] (-1.5,0.5)--(-1,0)  ;
\draw [line width=1.2pt ] (0,0)--(-1,0)  ;
\draw [line width=1.2pt ] (0,1)--(1,1)  ;
\draw [line width=1.2pt ] (0,0)--(1,0)  ;
\path (-2,-2)--(5,-2) node [pos=.1, right ] {The projective dimension is $27+2+2+4=35$.};

\end{tikzpicture}
\end{center}
\end{figure}
\end{example}

\end{document}